\documentclass[a4paper, 11pt]{article}

\usepackage[T1]{fontenc}
\usepackage[utf8]{inputenc}
\usepackage{lmodern}
\usepackage[english]{babel}
\usepackage[all,cmtip]{xy}
\usepackage{tkz-fct}
\usepackage{tikz}
\usetikzlibrary{positioning}
\usepackage{bm}%Pour des symboles en gras
\usepackage{enumitem}

\usepackage{amsmath, amsthm, comment, color, graphicx, hyperref}
\usepackage[capitalize,nameinlink,noabbrev]{cleveref}

\usepackage{amsfonts}
\usepackage{amssymb}

\usepackage{todonotes}

\theoremstyle{definition} %%% for statements in roman typeface

 \newtheorem{definition}{Definition}[section]
 \newtheorem{remark}[definition]{Remark}
 \newtheorem{example}[definition]{Example}
 \newtheorem{convention}[definition]{Convention}

\theoremstyle{plain}      %%% for statements in italic typeface

 \newtheorem{proposition}[definition]{Proposition}
 \newtheorem{theorem}[definition]{Theorem}
 
 \newtheorem{lemma}[definition]{Lemma}

\newtheorem*{theorem*}{Theorem}

\newtheorem{claim}{Claim}

\setlist[description]{leftmargin=0cm}

\newcommand{\eps}{\varepsilon}

\newcommand{\minus}{\setminus}

\newcommand{\G}{\Gamma}
\newcommand{\N}{\mathbf N}
\newcommand{\dX}{\partial X}
\newcommand{\bgamh}{\partial (\Gamma, \mathcal{H})}
\DeclareMathOperator{\Isom}{Isom}
\DeclareMathOperator{\Stab}{Stab}
\DeclareMathOperator{\CT}{CT}
\DeclareMathOperator{\Hom}{Hom}
\DeclareMathOperator{\Homeo}{Homeo}
\DeclareMathOperator{\Cay}{Cay}

\DeclareMathOperator{\SL}{SL}
\newcommand{\CP}{\mathbb{CP}}
\newcommand{\RP}{\mathbb{RP}}
\newcommand{\C}{\mathbb{C}}
\newcommand{\R}{\mathbb{R}}
\newcommand{\Z}{\mathbb{Z}}
\renewcommand{\H}{\mathbb{H}}

\title{Limits of limit sets in rank-one symmetric spaces}

\author{A. Guilloux\thanks{IMJ-PRG, OURAGAN, Sorbonne Université,
    CNRS, INRIA, \texttt{antonin.guilloux@imj-prg.fr}; work partially funded by ANR-23-CE40-0012-03}  \and
  T. Weisman\thanks{University of Michigan,
    \texttt{tjwei@umich.edu}. Partially supported by NSF grant
    DMS-2202770.}}

\begin{document}

\maketitle

\begin{abstract}
We consider the question of continuity of limit
    sets for sequences of geometrically finite subgroups of isometry
    groups of rank-one symmetric spaces, and prove analogues of
    classical (Kleinian) theorems in this context. In particular we
    show that, assuming strong convergence of the sequence of
    subgroups, the limit sets vary continuously with respect to
    Hausdorff distance, and if the sequence is weakly type-preserving,
    the sequence of Cannon-Thurston maps also converges uniformly to a
    limiting Cannon-Thurston map. Our approach uses the theory of
    extended geometrically finite representations, developed recently
    by the second author.
\end{abstract}

\section{Introduction}

Let $X$ be a noncompact rank-one symmetric space, for example real,
complex, or quaternionic hyperbolic space. Consider a finitely
generated group $\G$, and let $\rho: \G\to \Isom(X)$ be a discrete
faithful representation. Its image $\rho(\Gamma)$ has a limit set,
denoted $\Lambda_\rho$:
\begin{definition}[Limit set, see e.g. {\cite[Chapter 7]{DSU}}]
  The limit set $\Lambda_G$ of a subgroup $G\subset \Isom(X)$ is the
  set of accumulation points in $\partial X$ of any $G$-orbit in
  $X$. It is a compact subset of $\partial X$, and any $G$-invariant
  closed subset of $\Isom(X)$ with cardinality at least $2$ contains
  $\Lambda_G$.
\end{definition}

The limit set has classically received a great deal of attention in
this setting, see e.g. \cite{CorletteIozzi}.  When deforming $\rho$ to
a nearby representation $\rho'$, the limit set will vary. The
continuity of the assignment $\rho' \mapsto \Lambda_{\rho'}$ (with
respect to the Hausdorff topology on compact subsets of $\partial X$)
is not guaranteed, even in the classical setting where $X$ is
3-dimensional real hyperbolic space $\H^3_\R$. A crucial observation
is that the notion of convergence of representations with respect to
the compact-open topology on maps $\G \to \Isom(X)$ (usually called
\emph{algebraic convergence}) is not by itself enough to guarantee
convergence of the limit sets; the convergence $\rho_n \to \rho$
should also be \emph{geometric} (see
\Cref{def:strong_convergence}). The combination of algebraic and
geometric convergence is called \emph{strong} convergence.

When $X = \H^3_\R$, several different results relate the convergence
of limit sets $\Lambda_{\rho_n} \to \Lambda$ to strong convergence
$\rho_n \to \rho$; see e.g. \cite{JorgensenMarden,
  McMullen,AndersonCanary,Evans}, or refer to \cite{Marden} for a
survey. We prove the following result for general rank-one $X$, in the
situation where the limiting representation $\rho$ is
\emph{geometrically finite}:
\begin{theorem}[Convergence of limit sets]
  \label{thm:theorem_1}
  Let $\G$ be a finitely generated group, let $X$ be a rank-one
  symmetric space, and let $\rho:\G \to \Isom(X)$ be a faithful
  geometrically finite representation. Let $(\rho_n)$ be a sequence of
  faithful representations of $\G$ such that $\rho_n$ converges
  strongly to $\rho$.

  Then, for all sufficiently large $n$, $\rho_n$ is geometrically
  finite, and the limit sets $\Lambda_{\rho_n}$ converge to
  $\Lambda_\rho$ with respect to Hausdorff distance on $\partial X$. 
\end{theorem}

\begin{remark}\
  \begin{enumerate}[label=(\alph*)]
  \item When $X = \H_\R^3$, \Cref{thm:theorem_1} follows from work of
    J\o rgensen-Marden \cite{JorgensenMarden}, and when $X = \H^n_\R$,
    the theorem is covered by a result of McMullen \cite[Thm.
    4.1]{McMullen}. In both of these results, the hypothesis of strong
    convergence $\rho_n \to \rho$ can be relaxed to one of
    \emph{relative} strong convergence. We are also able to relax the
    hypothesis in this way; see
    \Cref{prop:strong_relative_convergence}.
  \item Again in the case $X = \H_\R^3$, J\o rgensen-Marden proved a
    partial converse to \Cref{thm:theorem_1}, and in this case
    McMullen also proved results implying continuity of the Hausdorff
    dimension of the sequence of limit sets under certain
    circumstances. We will not pursue either of these directions in
    this paper.
  \item The hypothesis in \Cref{thm:theorem_1} that each $\rho_n$ is
    faithful is likely unnecessary. Indeed, McMullen's result for
    $\H^n_\R$ does not make this assumption (although his proof does
    assume that each $\rho_n(\Gamma)$ is torsion-free). We expect to
    explore this case further in future work.
  \end{enumerate}
\end{remark}

\subsection{Cannon-Thurston maps}

In \cite{Mj-Series}, Mj-Series proved another version of continuity
for limit sets of sequences of geometrically finite representations in
$\H_\R^3$, in terms of uniform continuity of the associated
\emph{Cannon-Thurston maps}. In general, when $\Gamma_1 \to \Gamma_2$
is an isomorphism of Kleinian groups, a \emph{Cannon-Thurston map} is
a continuous equivariant map
$\Lambda_{\Gamma_1} \to \Lambda_{\Gamma_2}$ between the limit sets of
$\Gamma_1$ and $\Gamma_2$.

If $\Gamma_1$ and $\Gamma_2$ are arbitrary Kleinian groups, then such
a map is not guaranteed to exist. However, there is always a
Cannon-Thurston map $\Lambda_{\Gamma_1} \to \Lambda_{\Gamma_2}$ if
$\Gamma_1, \Gamma_2$ are geometrically finite and the isomorphism
$\Gamma_1 \to \Gamma_2$ is \emph{weakly type-preserving}, meaning
every parabolic element in $\Gamma_1$ is taken to a parabolic element
of $\Gamma_2$. One says that a sequence $(\rho_n)$ of faithful
geometrically finite representations $\Gamma \to \Isom(\H^3_\R)$ is
weakly type-preserving if each isomorphism
$\rho_n \circ \rho_{1}^{-1}$ is weakly type-preserving (see also
\Cref{defn:type_preserving}); in this case there is a sequence of
Cannon-Thurston maps $\Lambda_{\rho_1} \to \Lambda_{\rho_n}$ between
the limit sets of these representations.

We prove the following uniform continuity result for Cannon-Thurston
maps associated to geometrically finite representations:
\begin{theorem}[Convergence of Cannon-Thurston maps]
  \label{thm:theorem_2}
  Let $X$ be a rank-one symmetric space, let $\G$ be a finitely
  generated group, and let $(\rho_n)_{n \in \N}$ be a weakly
  type-preserving sequence of faithful geometrically finite
  representations, converging relatively strongly to a geometrically
  finite representation $\rho$.

  Then the sequence of Cannon-Thurston maps
  $\CT_{1,n}:\Lambda_{\rho_1} \to \Lambda_{\rho_n}$ exists and
  converges uniformly to a Cannon-Thurston map
  $\CT_{1,\infty}:\Lambda_{\rho_1} \to \Lambda_\rho$.
\end{theorem}

In the special case $X = \H^3_\R$, \Cref{thm:theorem_2} exactly
recovers the aforementioned result of Mj-Series (see
\cite[Thm. A]{Mj-Series}).

\begin{remark}
  One consequence of our proof of \Cref{thm:theorem_1} will be that,
  if $(\rho_n)_{n \in \N}$ is a sequence of faithful representations
  converging strongly to a geometrically finite representation $\rho$,
  then a subsequence of $(\rho_n)$ is weakly type-preserving (see
  \Cref{prop:subseq_type_preserve}). Thus the context of
  \Cref{thm:theorem_2} is no more restricted than that of
  \Cref{thm:theorem_1}.
\end{remark}

\subsection{Proof strategy}

Our proof of \Cref{thm:theorem_1} and \Cref{thm:theorem_2} does not
use conformal dynamics, but rather relies on the notion of
\emph{extended geometrically finite} (EGF) representations and
\emph{peripherally stable} deformations developed by the second author
in \cite{Weisman-EGF}. Interestingly, these notions were originally
developed to deal with generalizations of geometrical finiteness in
higher-rank semisimple Lie groups, but we show in this paper that the
techniques also apply fruitfully in rank one. In particular, the
approach gives an alternative proof of the classical (Kleinian)
versions of our main theorems.

The organisation of the paper is as follows. In the following
\Cref{sec:GF-EGF} we present the definition of EGF representations in
the context of a rank-one symmetric space $X$, and explain the
relation to the usual definition of geometrical finiteness. Even in
the rank-one case, the two notions are not exactly identical, but they
are close enough to be essentially equivalent for our concerns.

In \Cref{sec:peripheral}, we discuss deformations of EGF
representations, and present the notion of a peripherally stable
subspace. Here we recall the main theorem of \cite{Weisman-EGF} (see
\Cref{thm:relative_stability_theorem}), which states that small
deformations of EGF representations in peripherally stable subspaces
are still EGF, and that their limit sets deform (semi-)continuously in
these subspaces. In this section we also briefly explain the
connection between peripheral stability and a \emph{relative
  automaton}, an important technical tool developed in
\cite{Weisman-EGF} which is useful in later sections of this paper.

In \Cref{sec:strong-peripheral}, we give some reminders about the
Chabauty topology, and then prove the equivalence between peripheral
stability, strong convergence, and relative strong convergence
(\Cref{prop:strong_relative_convergence}). This is the most technical
section of the paper; once we have established this equivalence, we
are able to give quick proofs of \Cref{thm:theorem_1} and
\Cref{thm:theorem_2} in \Cref{ssec:convergencetheorems}, as
corollaries of \Cref{thm:relative_stability_theorem}.

\section{Geometrical finiteness and extended geometrical finiteness}\label{sec:GF-EGF}

%The goal of this section is to review the notion of \emph{extended geometrically finite} from \cite{Weisman-EGF} and to understand its meaning in the case of isometries of a rank one negatively curved symmetric space.

Let $\G$ be a finitely generated group and let $\rho:\G \to \Isom(X)$
be a geometrically finite representation. In this situation, $\G$ is a
relatively hyperbolic group (see \cite{Bowditch-RelHypGroups}),
relative to its collection $\mathcal{H}$ of maximal parabolic
subgroups, i.e. the set of stabilizers of points in $\partial X$ fixed
by a parabolic isometry in $\rho(\G)$. We say that $(\G, \mathcal{H})$
is a \emph{relatively hyperbolic pair}. Recall that such a pair is
equipped with its \emph{Bowditch boundary} $\partial(\G,\mathcal H)$,
which in this situation is equivariantly homeomorphic to the limit set
of $\rho(\G)$ (see the beginning of \Cref{sec:egf_gf_rep}).

The second author developed the notion of extended geometrical
finiteness (or EGF) in \cite{Weisman-EGF} for representations of a
general relatively hyperbolic group $(\G,\mathcal H)$ into a
semisimple Lie group $G$. One goal of this section is to prove that,
when $G$ is a rank-one Lie group, and all of the peripheral subgroups
$H\in\mathcal H$ of $\Gamma$ are virtually nilpotent, then EGF
representations of $\G$ are precisely the same thing as geometrically
finite representations (see \Cref{prop:egf_nilpotent_gf} below). We
will also explain the connection between the equivariant homeomorphism
$\partial(\G, \mathcal{H}) \simeq \Lambda_\rho$ and the objects
appearing in the definition of an EGF representation.

In general it is possible to construct EGF representations which are
not geometrically finite (see
\Cref{ex:geometrically_infinite_egf}). However, in the context of this
paper, the peripheral subgroups of our relatively hyperbolic group
$\G$ will always be virtually nilpotent, and so EGF representations
and geometrically finite representations will be equivalent. The main
reason we work with EGF representations is that they come equipped
with some extra structure which turns out to be very useful when
considering the behavior of deformations of geometrically finite
groups.

\subsection{EGF representations in rank-one symmetric spaces}

Below, we give the definition of an EGF representation
$\rho:\Gamma \to \Isom(X)$ when $\Gamma$ is a relatively hyperbolic
group and $X$ is a rank-one symmetric space. The full definition of an
EGF representation into an arbitrary-rank semisimple Lie group $G$ is
more complicated. That level of generality is not relevant for this
paper, so we refer the interested reader to \cite{Weisman-EGF} for
details.

\begin{definition}
  \label{defn:egf_definition}
  Let $(\Gamma, \mathcal H)$ be a relatively hyperbolic pair, with
  Bowditch boundary $\bgamh$. A representation
  $\rho:\Gamma \to \Isom(X)$ is \emph{extended geometrically finite}
  if there is a closed $\rho$-invariant subset $\Lambda \subset \dX$
  and a $\rho$-equivariant map $\phi:\Lambda \to \bgamh$ (called a
  \emph{boundary extension}) satisfying the following condition:

  For any sequence $\gamma_n \in \G$ satisfying
  $\gamma_n^{\pm 1} \to z_{\pm} \in \bgamh$, any compact subset
  $K \subset \dX \minus \phi^{-1}(z_-)$, and any open subset
  $U \subset \dX$ containing $\phi^{-1}(z_+) \subset U$, we have
  $\rho(\gamma_n)K \subset U$ for all sufficiently large $n$.
\end{definition}

\begin{remark}
  The Bowditch boundary of a relatively hyperbolic group is really an
  invariant of a relatively hyperbolic \emph{pair}, i.e. a relatively
  hyperbolic group together with a choice of a collection of
  peripheral subgroups (which is not always uniquely determined). So
  the definition of an EGF representation depends on the choice of
  collection $\mathcal H$. This is important in this paper, since
  sometimes we will want to consider different peripheral structures
  on the same group.
\end{remark}

\Cref{defn:egf_definition} is our primary definition of
an EGF representation in this paper because it is fairly close to the
original definition given in \cite{Weisman-EGF}. However, in our
current setting, there is an alternative formulation of the definition
given in terms of convergence group actions. Recall that, if $M$ is a
Hausdorff space, an action $\Gamma \to \Homeo(M)$ is a \emph{discrete
  convergence action} if, for every sequence of pairwise distinct
elements $\gamma_n \in \Gamma$, one can extract a subsequence and find
points $a, b \in M$ so that the restrictions
$\gamma_n|_{M \minus \{a\}}$ converge to the constant map $b$,
uniformly on compacts in $M \minus \{a\}$.

\begin{proposition}
  \label{prop:egf_convergence_definition}
  Suppose that $(\Gamma, \mathcal H)$ is a relatively hyperbolic
  pair. A representation $\rho:\Gamma \to \Isom(X)$ is EGF if and only
  if there is a closed invariant set $\Lambda \subset \dX$ and a
  $\rho$-equivariant map $\phi:\Lambda \to \bgamh$ such that the
  induced action on the space $\dX/\sim_{\phi}$ is a discrete
  convergence action, where $\sim_{\phi}$ is the equivalence relation
  identifying points in the same fiber of $\phi$.
\end{proposition}

This proposition is a straightforward consequence of the fact that any
discrete subgroup of $\Isom(X)$ acts as a discrete convergence group
on $\dX$ (see \cite{Tukia2}).

\subsection{EGF representations and geometrically finite
  representations}\label{sec:egf_gf_rep}
In sections 6 and 9 of \cite{Bowditch-RelHypGroups}, Bowditch uses the
Beardon-Maskit definition of geometrical finiteness (see
\cite{BeardonMaskit}, \cite{BowditchGF}) to prove the following facts
about any geometrically finite representation of $\Gamma$:
  \begin{itemize}
  \item If $\mathcal H$ is the collection of maximal subgroups of $\G$
    sent by $\rho$ to parabolic subgroups of $\Isom(X)$, then
    $(\Gamma,\mathcal H)$ is a relatively hyperbolic pair.
  \item The limit set of $\rho(\G)$ is canonically homeomorphic to the
    Bowditch boundary $\bgamh$ of the pair.
\end{itemize}
\Cref{prop:egf_convergence_definition} makes it clear that any
geometrically finite representation is also an EGF representation for
this relatively hyperbolic structure: the closed invariant set
$\Lambda$ can be taken to be the limit set of the group
$\rho(\Gamma)$, and $\phi:\Lambda\to \bgamh$ is given by the canonical
homeomorphism. Since there is a preferred choice of peripheral
structure in this case, we make the following definition:

\begin{definition}
  For any representation $\rho : \G\to \Isom(X)$ of a finitely generated group, the collection of the
  \emph{$\rho$-parabolic subgroups} is the collection of maximal
  subgroups of $\G$ sent by $\rho$ to parabolic subgroups of
  $\Isom(X)$.
\end{definition}

With this terminology, if $\rho$ is a geometrically finite
representation, then it is also EGF with respect to the
$\rho$-parabolic subgroups, with a homeomorphic boundary extension. In
fact, the converse also holds:
\begin{theorem}[{See \cite[Theorem 1.10]{Weisman-EGF}}]
  \label{thm:injective_bdry_extension}
  Let $\rho:\G \to \Isom(X)$ be a representation of a finitely
  generated group, and let $\mathcal H$ the collection of
  $\rho$-parabolic subgroups. Then $\rho$ is geometrically finite if
  and only if $(\Gamma,\mathcal H)$ is a relatively hyperbolic pair,
  and $\rho$ is an EGF representation with an injective boundary
  extension $\phi:\Lambda \to \bgamh$. In this case, the set $\Lambda$
  is the limit set of $\rho(\Gamma)$.
\end{theorem}

On the other hand, it is in general possible (even in the rank-one
setting) to construct EGF representations whose boundary extensions
are \emph{not} injective. This can occur for two different reasons,
which we cover below.

The first reason that injectivity may fail is essentially an artifact
of the definition: the closed invariant set $\Lambda$ is \emph{not}
uniquely determined by the conditions given in
\Cref{defn:egf_definition}. If $\rho(\Gamma)$ is non-elementary, then
$\Lambda$ must always contain the limit set of $\rho(\Gamma)$, but it
is sometimes also possible to find a larger set which still satisfies
the definition.
\begin{example}
  Consider a geometrically finite Fuchsian group
  $\G\subset \SL(2,\R)\subset\SL(2,\C)$ which is not convex cocompact
  and hence has parabolic elements; $\SL(2,\Z)$ does the trick. Pick
  an invariant system of disjoint horoballs in $\H^2$ centered at the
  parabolic points. Now, identify
  $\H^2\subset \partial \H^3\simeq \CP^1$ with the upper half plane in
  $\C$. Let $\Lambda$ be the union of $\RP^1$ and the horoballs, and
  let $\phi:\Lambda \to \RP^1$ be the map which sends horoballs to
  their center and non-cusp points to themselves.  One may then check
  that $\phi$ is a boundary extension, but $\Lambda$ is not the limit
  set of $\G$ (which is $\RP^1$).
\end{example}

Fortunately, we can safely brush this particular issue aside in the
rank-one setting: it turns out that it always possible to choose the
invariant set $\Lambda$ in \Cref{defn:egf_definition} to be the limit
set of $\rho(\Gamma)$ (see \Cref{prop:good_boundary_extension} below).

Even with this assumption, though, it is still possible that the
boundary extension $\phi$ could fail to be injective. This is a
feature, rather than a bug, in the definition of an EGF
representation: it allows us to take into account the fact that the
chosen collection $\mathcal H$ of peripheral subgroups in $\G$ may
\emph{not} be precisely the same as the collection of $\rho$-parabolic
subgroups. This is especially relevant when we consider deformations
of geometrically finite representations which do \emph{not} preserve
the natural choice of peripheral subgroups, and we want to define
Cannon-Thurston maps between the limit sets of these representations
(see \Cref{sec:ct_maps}).

In some cases, however, even the ``most natural'' choice of peripheral
subgroups gives rise to a boundary extension for an EGF representation
which is still not injective. This will occur precisely when the image
of the EGF representation is a geometrically infinite group.

\begin{example}
  \label{ex:geometrically_infinite_egf}
  Consider a finitely generated geometrically infinite discrete group
  $\Gamma \subset \mathrm{PO}(3,1)$. By including $\mathrm{PO}(3,1)$
  into $\mathrm{PO}(4,1)$, we may view $\Gamma$ as a geometrically
  infinite discrete group preserving an isometrically embedded $\H^3$
  in $\H^4$; the limit set of $\Gamma$ in $\partial \H^4$ is contained
  in the embedded 2-sphere at the boundary of this $\H^3$. Then let
  $\Gamma'$ be a conjugate of $\Gamma$ in $\mathrm{PO}(4,1)$ by some
  isometry taking this 2-sphere completely off of itself.

  Using Klein-Maskit ``ping-pong'' combination theorems (see
  e.g. \cite{Maskit}), one can show that, possibly after replacing
  $\Gamma$ and $\Gamma'$ with finite-index subgroups, the subgroup
  $\langle \Gamma, \Gamma' \rangle \subset \mathrm{PO}(4,1)$ is
  naturally isomorphic to the abstract free product
  $\Gamma * \Gamma'$. The free product is a relatively hyperbolic
  group, relative to the collection of conjugates of
  $\Gamma, \Gamma'$, and one may use the same ping-pong techniques to
  prove that the limit set of $\langle \Gamma, \Gamma' \rangle$
  surjects equivariantly onto the Bowditch boundary of the free
  product, so that the preimage of each parabolic point is the limit
  set of some conjugate of $\Gamma$ or $\Gamma'$. This gives the
  boundary extension for an EGF representation with geometrically
  infinite image.
\end{example}

Since the example above is geometrically infinite,
\Cref{thm:injective_bdry_extension} says that there is no choice of
injective boundary extension for this representation. However, even in
this case, the boundary extension is well-behaved in the sense that
all the non-injectivity occurs at the peripheral subgroups. This is
true in general, due to the following consequence of \cite[Proposition
4.8]{Weisman-EGF}:
\begin{proposition}
  \label{prop:good_boundary_extension}
  Let $(\Gamma, \mathcal H)$ be a relatively hyperbolic pair, let
  $\rho:\G \to \Isom(X)$ be a non-elementary discrete representation,
  and let $\Lambda_\rho$ be the limit set of $\rho(\Gamma)$.

  If $\rho$ is an EGF representation, then there is a unique boundary
  extension $\phi:\Lambda_\rho \to \bgamh$% for $\rho$ which is defined
  %on the limit set $\Lambda_\rho$
  . Moreover, for this boundary
  extension:
  \begin{enumerate}
  \item If $z \in \bgamh$ is a conical limit point, then the fiber
    $\phi^{-1}(z)$ is a singleton;
  \item If $z \in \bgamh$ is a parabolic point, then $\phi^{-1}(z)$ is
    the limit set of $\rho(\Stab_{\G}(z))$.
  \end{enumerate}
\end{proposition}
\begin{proof}
  Since $\rho(\Gamma)$ is non-elementary, any closed invariant subset
  of $\dX$ contains $\Lambda_\rho$, so in particular
  $\Lambda_\rho \subset \Lambda$ for any subset $\Lambda$ as in
  \Cref{defn:egf_definition}.

  Proposition 4.8 in \cite{Weisman-EGF} asserts that one can choose $\Lambda$ and the boundary extension $\phi:\Lambda \to \bgamh$ so that, for
  every conical limit point $z \in \bgamh$, $\phi^{-1}(z)$ is a
  singleton. This implies that each such $\phi^{-1}(z)$ is in
  $\Lambda_\rho$: we can always find a sequence $\gamma_n \in \G$ so
  that $\gamma_n \to z$ in the compactification $\G \cup
  \bgamh$. Then, \Cref{defn:egf_definition} implies that for some
  nonempty open subset $U \subset \dX$, the sequence
  $\rho(\gamma_n)\overline{U}$ converges to the singleton
  $\phi^{-1}(z)$ and so this singleton must lie in $\Lambda_\rho$.

  Proposition 4.8 in \cite{Weisman-EGF} also asserts that the boundary
  extension above can be chosen so that, for each parabolic point
  $p \in \bgamh$, the fiber $\phi^{-1}(p)$ consists of accumulation
  points of sequences of a particular form. Precisely, the proposition
  states that there is an open subset $C_p \subset \dX$ so that
  $\phi^{-1}(p)$ is exactly the closure of the set of accumulation
  points of sequences $\rho(\gamma_n)x$, for $x \in C_p$ and
  $\gamma_n$ a sequence of pairwise distinct elements in
  $\Stab_\G(p)$. Now, any such accumulation point must lie in the limit set of
  $\rho(\Stab_\G(p))$.

  %But, any such accumulation point must lie in the limit set of
  %$\rho(\Stab_\G(p))$: if $x$ lies in the limit set of
  %$\rho(\Stab_\G(p))$, then this follows because the limit set is
  %closed and $\Stab_\G(p)$-invariant, and otherwise $x$ lies in the
  %domain of discontinuity for $\Stab_\G(p)$ and $\rho(\gamma_n)x$ can
  %still only accumulate in the limit set for this subgroup.

  Every point in $\bgamh$ is either a conical limit point or a
  parabolic point. So, the two cases above show that the fiber above
  every point in $\bgamh$ is contained in $\Lambda_\rho$, and that
  this fiber is completely determined by the representation $\rho$ and
  satisfies the conditions in the statement of the proposition.
\end{proof}

\begin{convention}
  Since the boundary extension determined by
  \Cref{prop:good_boundary_extension} is unique, for the rest of this
  paper, we will always refer to ``the'' boundary extension for an EGF
  representation when we mean the extension determined by the
  proposition.
\end{convention}

\begin{remark}
  In light of \Cref{prop:good_boundary_extension}, one may well ask
  why the definition of an EGF representation allows for different
  boundary extensions for the same representation of the same group
  with the same peripheral structure. The answer is that the
  uniqueness property in \Cref{prop:good_boundary_extension} is more
  subtle in the higher-rank setting, which makes it harder to
  determine a ``best'' choice for the set $\Lambda$ in the definition.
\end{remark}

\subsection{EGF representations with nilpotent peripheral subgroups}

If a representation $\rho:\Gamma \to \Isom(X)$ is geometrically
finite, then every parabolic subgroup of $\G$ maps to a discrete
subgroup of $\Isom(X)$ fixing a unique point in $\dX$, and is thus
virtually nilpotent.

\begin{proposition}
  \label{prop:egf_nilpotent_gf}
  Let $(\G, \mathcal{H})$ be a relatively hyperbolic pair, such that
  each $H \in \mathcal{H}$ is virtually nilpotent. For any
  representation $\rho:\G \to \Isom(X)$, the following are equivalent:
  \begin{enumerate}[label=(\roman*)]
  \item \label{item:gf_parabolics} The representation
    $\rho:\G \to \Isom(X)$ is geometrically finite, and $\mathcal{H}$
    contains the collection of $\rho$-parabolic subgroups of
    $\G$.
  \item \label{item:egf_h} The representation $\rho$ is EGF with
    respect to $\mathcal{H}$.
  \end{enumerate}
\end{proposition}
\begin{proof}
  First we prove \ref{item:gf_parabolics} $\implies$ \ref{item:egf_h}.
  \Cref{thm:injective_bdry_extension} implies that a geometrically
  finite representation is always extended geometrically finite with
  respect to its collection $\mathcal{H}'$ of parabolic
  subgroups. Then, as $\mathcal{H}' \subseteq \mathcal{H}$, we may
  apply a special case of a relativization theorem of Wang
  \cite[Thm. 1.8]{wang} to see that $\rho$ is also EGF with respect to
  $\mathcal{H}$.
  
  Next, we prove \ref{item:egf_h} $\implies$
  \ref{item:gf_parabolics}. Suppose that $\rho:\G \to \Isom(X)$ is EGF
  with respect to $\mathcal{H}$, with boundary extension
  $\phi:\Lambda \to \bgamh$. The image of each peripheral subgroup
  $H \in \mathcal H$ is a virtually nilpotent discrete subgroup in
  $\Isom(X)$, which means that its limit set $\Lambda(\rho(H))$
  consists of either zero, one, or two points. If the limit set
  contains zero points then $H$ is finite, which is disallowed by the
  definition of an EGF representation (our convention is that all
  peripheral subgroups of a relatively hyperbolic group are
  infinite). So, we know that each $\Lambda(\rho(H))$ contains either
  one or two points.

  If the limit set of $H$ contains two points, then $H$ is virtually
  cyclic and $\rho(H)$ is quasi-isometrically embedded, i.e. convex
  cocompact. Put another way, the restriction of $\rho$ to $H$ is
  geometrically finite with respect to an \emph{empty} collection of
  cusp subgroups of $H$, so we can apply the other direction of the
  relativization theorem cited above (see also \cite[Thm
  1.15]{Weisman-EGF}) to see that $\rho$ is also an EGF representation
  with respect to the peripheral structure
  \[
    \mathcal H' := \{H \in \mathcal H : |\Lambda(\rho(H))| = 1\}.
  \]
  
  \Cref{prop:good_boundary_extension} then implies that the boundary
  extension with respect to this peripheral structure is
  injective. Thus, by \Cref{thm:injective_bdry_extension}, $\rho$ is
  geometrically finite, and the collection
  $\mathcal{H}' \subseteq \mathcal{H}$ is precisely the collection of
  parabolic subgroups of $\G$.
\end{proof}

% \begin{remark}\label{remark:change_peripheral_structure}
%   The proof above shows that, when $(\Gamma, \mathcal{H})$ is a
%   relatively hyperbolic pair and every group in $\mathcal{H}$ is
%   virtually nilpotent, then one may obtain an injective boundary
%   extension for an EGF representation $\Gamma \to \Isom(X)$ by lifting
%   the original boundary extension along the modified peripheral
%   structure $\mathcal{H}'$. Specifically, if
%   $\phi:\Lambda \to \partial(\Gamma, \mathcal{H})$ is the original
%   boundary extension, and
%   $\phi':\Lambda \to \partial(\Gamma, \mathcal{H}')$ is the injective
%   (homeomorphic) boundary extension, then the composition
%   $\pi = \phi \circ (\phi')^{-1}$ is a $\Gamma$-equivariant projection
%   $\partial (\Gamma, \mathcal{H}') \to \partial(\Gamma,
%   \mathcal{H})$. The map $\pi$ is determined only by the abstract
%   group $\Gamma$ and the peripheral subgroups in $\mathcal{H}$,
%   $\mathcal{H}'$.
% \end{remark}

We have now finished our review of most important properties of EGF
representations and their relation to geometrical finiteness. However,
to prove some of the results in this paper, we need to work with an
additional technical tool: a \emph{relative automaton}.

\subsection{The relative automaton for an (extended) geometrically
  finite representation}\label{sec:automaton_exists}

The automaton discussed in this section is constructed (in a more
general setting) in \cite{Weisman-EGF}, building upon ideas and
results going back to the computational approach to hyperbolic
groups. Below we will state several results regarding the
construction. All of these results also apply in the more general
context, but we will just state versions appropriate for the present
setting.

Fix a finitely generated group $\G$, let $\rho:\G \to \Isom(X)$ be a
geometrically finite representation and consider $\mathcal H$ the
collection of $\rho$-parabolic subgroups.  Let $\Lambda_\rho$ denote
the limit set of $\rho(\Gamma)$. There are finitely many orbits of
parabolic points in $\Lambda_\rho$, so we fix once and for all a
finite subset $\Pi \subset \Lambda_\rho$, containing exactly one point
from each of these parabolic orbits. We also fix a metrization of
$\partial X$; this can be taken to be a visual metric, but the precise
choice does not matter.

\begin{definition}
  A \emph{relative automaton} associated to $\rho$ consists of
  the following data:
  \begin{itemize}
  \item a finite directed graph $\mathcal G$, whose vertex set $Z$ is
    a subset of the limit set of $\G$;
  \item a pair of mappings $z \mapsto W(z)$ and $z \mapsto L(z)$
    defined on $Z$, where $W(z)$ is an open subset of $\dX$ and $L(z)$
    (the \emph{label set}) is a subset of $\G$.
  \end{itemize}
\end{definition}

The main result of Section~6 of \cite{Weisman-EGF} tells us that, for
the geometrically finite representation $\rho$, it is always possible
to construct a relative automaton for $\rho$ which satisfies all of
the following properties:
\begin{enumerate}[label=(A\arabic*)]
\item\label{item:automaton_sets_proper} The closure of each subset
  $W(z)$ is a proper subset of $\dX$.
\item\label{item:automaton_edge_inclusion} There is a fixed
  $\eps > 0$ so that, for each directed edge $z \to y$ in
  $\mathcal G$ and each $\alpha \in L(z)$, we have an inclusion
  \[
    \rho(\alpha) \overline{N_\eps(W(y))} \subset W(z).
  \]
\item\label{item:parabolic_vertex_coset} If $z \in Z$ is a parabolic
  point (so that $z = \rho(g)p$ for some $p \in \Pi$ and $g\in \G$), 
  then $L(z)$ is a subset of the coset $g\Stab_\G(p)$.
\item\label{item:conical_vertex_singleton} If $z \in Z$ is
  \emph{not} a parabolic point, then the label set $L(z)$ is a
  singleton.
\item\label{item:parabolic_edge_compatible} For every edge $z \to y$
  in $\mathcal G$, if $z$ is a parabolic point, equal to $\rho(g)p$ for
  $p \in \Pi$, then $W(z)$ contains $z$, and $\overline{W(y)}$ does
  not contain $p$.
\item\label{item:automaton_quasidensity} There is a uniform constant
  $R > 0$ so that, for each element $\gamma \in \Gamma$, we can find
  a directed vertex path $z_1 \to \ldots \to z_{n+1}$ in
  $\mathcal G$ and elements $\alpha_i \in L(z_i)$ so that the
  product
  \[
    \alpha_1 \cdots \alpha_n
  \]
  lies within distance $R$ of $\gamma$. Here, $\G$ is equipped with
  the word metric induced by some fixed choice of finite generating
  set.
\end{enumerate}

\begin{remark}
  The reader may also wish to refer to \cite[Section 3]{MMW2} for a
  somewhat simpler version of the construction in \cite{Weisman-EGF}
  in a slightly different context.
\end{remark}

The relative automaton above contains all the information needed to
reconstruct the limit set of $\rho$. The use of automata to encode
limit sets of Kleinian groups has a long history, tracing back to
Sullivan's original ``symbolic coding'' argument for structural
stability of convex cocompact groups \cite{Sullivan1985}. The same
idea has also been used to compute visualizations of limit sets of
Kleinian groups; see e.g. \cite{McShaneParkerRedfern}. The
\emph{relative} automaton we consider here is very convenient to work
with when deforming $\rho$ in the space of EGF representations. We
will use it in the next section, where we deal with properties of this
deformation space.

\section{Deformations of EGF representations and peripheral stability}
\label{sec:peripheral}

The previous section dealt with the connection between the concepts of
extended geometrical finiteness and geometrical finiteness for a
single representation $\rho$. Now we want to understand families of
such representations, so we will review a key \emph{relative
  stability} property of EGF representations. Roughly, this property
says that if $\rho'$ is a small deformation of an EGF representation
$\rho:\G \to \Isom(X)$, and the restriction of $\rho'$ to peripheral
subgroups satisfies a certain condition, then $\rho'$ is also
EGF. This technical condition is called \emph{peripheral stability}.

As in the previous section, we will define a version of peripheral
stability which makes sense for a geometrically finite representation
$\rho:\G \to \Isom(X)$ when $X$ is a rank one symmetric space. This
will be simpler than the full definition when $\rho$ is an EGF
representation into a general semisimple Lie group; the definitions
are equivalent in the present context.

Note that, although we consider several notions of convergence for
representations throughout this text, we will always understand the
space $\Hom(\G, \Isom(X))$ to be equipped with the algebraic (or
compact-open) topology.

\subsection{Peripheral stability}

Let $\G$ be a finitely generated group and let $\rho:\G \to \Isom(X)$
be a faithful and geometrically finite representation. Let
$\mathcal{P} \subset \dX$ be the collection of all cusp points for
$\rho$, meaning that the collection $\mathcal{H}$ of $\rho$-parabolic
subgroups is precisely the set $\{\Stab_\G(p), p\in\mathcal P\}$.

\begin{definition}\label{def:peripheral_stability}
  A subset $O \subseteq \Hom(\G, \Isom(X))$ is \emph{peripherally
    stable} about $\rho$ if the following holds:

  Let $p \in \mathcal P$, let $U$ be a neighborhood of $p$ in $\dX$,
  let $K \subset \dX$ be a compact subset of $\dX \minus \{p\}$, and
  let $F$ be a finite subset of $\Stab_\G(p)$ such that
  \[
    \rho(\Stab_\G(p) \minus F)K \subset U.
  \]
  Then there is a relatively open subset $O'$ of $O$ (depending on
  $U$, $F$, and $K$) such that for all $\rho' \in O'$, we have
  \[
    \rho'(\Stab_\G(p) \minus F)K \subset U.
  \]
\end{definition}

In the current context, the relative stability property for EGF
representations can be stated as follows. For the result below, fix an
arbitrary metrization of
$\dX$.%, and let $\dhaus$ denote the Hausdorff
%distance with respect to this metric.
\begin{theorem}[{See \cite[Theorem 1.4]{Weisman-EGF}}]
  \label{thm:relative_stability_theorem}
  Let $\rho:\G \to \Isom(X)$ be a geometrically finite representation,
  let $\mathcal{H}$ be the associated collection of $\rho$-parabolic
  subgroups, and let $\phi:\Lambda \to \bgamh$ be the associated
  boundary extension. Let $O \subseteq \Hom(\G, \Isom(X))$ be a
  peripherally stable subspace about $\rho$.

  Then, for any $\eps > 0$ and any compact subset $Z \subset \bgamh$,
  there is a relatively open subset $O' \subset O$ satisfying the
  following: if $\rho' \in O'$, then $\rho'$ is also an EGF
  representation with boundary extension $\phi':\Lambda' \to \bgamh$
  satisfying
  \[
    (\phi')^{-1}(Z) \textrm{ lies inside an }\epsilon\textrm{-neighborhood of } \phi^{-1}(Z).
  \]

  Moreover, the set $\Lambda'$ can be taken to be the limit set of
  $\rho'(\Gamma)$, so that $\phi'$ is the unique EGF boundary
  extension described by \Cref{prop:good_boundary_extension}.
\end{theorem}

\begin{remark}
  The ``moreover'' part of \Cref{thm:relative_stability_theorem} is
  not stated explicitly as part of the cited result in
  \cite{Weisman-EGF}. However, this statement follows directly from
  the proof in that paper. Indeed, in \cite{Weisman-EGF}, the
  construction of the unique limit set described in
  \Cref{prop:good_boundary_extension} is carried out by applying the
  relative stability theorem to the (trivially) peripherally stable
  subspace $\{\rho\} \subset \Hom(\G, \Isom(X))$; see \cite[Remark
  9.18]{Weisman-EGF}.
\end{remark}

\subsection{Peripheral stability and relative automata}

The next lemma translates the peripheral stability condition into a
stability property for the relative automaton discussed previously in
\Cref{sec:automaton_exists}. This is another way to motivate the
definition of peripheral stability, since the automaton is a key tool
used both for the proof of \Cref{thm:relative_stability_theorem} in
\cite{Weisman-EGF} and for some results later in this paper.
\begin{lemma}
  \label{lem:peripheral_edge_stability}
  Let $O \subseteq \Hom(\G, \Isom(X))$ be a peripherally stable
  subspace about $\rho$, and let $\mathcal G$ be a relative
  automaton satisfying the properties listed above.

  There is an open neighborhood $O' \subseteq O$ of $\rho$ and
  a constant $\eps' > 0$ such that, for every $\rho' \in O'$, every
  directed edge $z \to y$ in $\mathcal{G}$, and every $\alpha \in
  L(z)$, we have
  \[
    \rho'(\alpha)\overline{N_{\eps'}(W(y))} \subset W(z).
  \]
\end{lemma}
\begin{proof}
  Property \ref{item:automaton_edge_inclusion} says that the desired
  inclusions are all satisfied when $\rho' = \rho$. So we just
  need to check that the desired condition is relatively open in $O$
  for each edge in $\mathcal G$, since there are finitely many
  such. First, if $z$ is not a parabolic point in $\Lambda_\rho$,
  then $L(z)$ is a singleton by \ref{item:conical_vertex_singleton}
  and so the condition is already open in $\Hom(\G, \Isom(X))$.

  On the other hand, if $z$ is a parabolic point, then by property
  \ref{item:parabolic_vertex_coset}, each $\alpha \in L(z)$ can be
  written $\alpha = g_z\alpha'$, with $\alpha' \in \Stab_\G(p)$ for
  $p \in \Pi$ and $g_z \in \G$ depending only on $z$. (Here
  $\Stab_\G(p)$ is the stabilizer with respect to the $\rho$
  action.)

  By property \ref{item:parabolic_edge_compatible}, the set $W(z)$
  contains $z$, which means that $\rho(g_z^{-1})W(z)$ contains $p$. By
  property \ref{item:parabolic_edge_compatible}, we know that
  $\overline{W(y)}$ does not contain $p$; this means that there is
  also some $\eps' > 0$ so that $\overline{N_{\eps'}(W(y))}$ does not
  contain $p$. Since there are only finitely many edges $z \to y$,
  this $\eps'$ can be chosen independently of $z$; we can also assume
  that $\eps'$ is smaller than the constant $\eps$ from condition
  \ref{item:automaton_edge_inclusion}.

  As $\Stab_\G(p)$ acts properly discontinuously on
  $\dX \minus \{p\}$, for all but finitely many
  $\alpha' \in \Stab_\G(p)$ we have
  $\rho(\alpha')\overline{N_{\eps'}(W(y))} \subset
  \rho(g_z^{-1})W(z)$. The peripheral stability assumption then
  implies that there is an open neighborhood $O'$ of $\rho$ in $O$, so
  that for all but finitely many $\alpha' \in \Stab_\G(p)$, every
  $\rho' \in O'$ satisfies
  \[
    \rho'(\alpha')\overline{N_{\eps'}(W(y))} \subset
    \rho'(g_z^{-1})W(z).
  \]
  This means that for all but finitely many exceptional
  $\alpha \in L(z)$, every $\rho' \in O'$ satisfies
  \begin{equation}
    \label{eq:perturbed_edge_inclusion}
    \rho'(\alpha)\overline{N_{\eps'}(W(y))} \subset W(z).
  \end{equation}
  However we also know that
  $\rho(\alpha)\overline{N_{\eps'}(W(y))} \subset W(z)$ for every
  $\alpha \in L(z)$, so by further shrinking $O'$ we can also ensure
  that for every $\rho' \in O'$, \eqref{eq:perturbed_edge_inclusion}
  holds for the finitely many exceptional $\alpha$ as well.
\end{proof}

\section{Strong convergence and
  peripheral stability}\label{sec:strong-peripheral}

As noted in the introduction, in the classical (Kleinian) context,
algebraic convergence of a sequence of representations
$\rho_n:\G \to \Isom(X)$ does not guarantee that the sequence of limit
sets $\Lambda_{\rho_n}$ converges to the limit set of the limiting
representation. Typically one must also assume that the sequence
$(\rho_n)$ also converges \emph{strongly}, meaning that the sequence
of subgroups $\rho_n(\G)$ converges in the Chabauty topology on the
space of closed subgroups of $\Isom(X)$.

In this section, after briefly reviewing the notion of strong
convergence for geometrically finite representations, we prove that it
is consistent with the notion of peripheral stability for EGF
representations.
% We work in a similar setting as in the previous sections: we let $\G$
% be a finitely generated group and let $\rho: \G\to \Isom(X)$ be a
% faithful and geometrically finite representation. In this situation,
% we may consider the collection of cusp points $\mathcal P\subset \dX$,
% together with the collection of
% $\mathcal H = \{\Stab_\G(p), p\in\mathcal P\}$ of their stabilizer
% subgroups in $\G$. Any $H \in \mathcal H$ is mapped to a parabolic
% subgroup of $\Isom(X)$ by $\rho$.

\subsection{Chabauty topology and strong convergence}

We refer to \cite{BHK} (see also \cite{delaharpe2008spaces} and
\cite[Section E.1]{BenedettiPetronio}) for a general reference on the
Chabauty topology.

Let $\mathcal{CL}:=\mathcal{CL}(\Isom(X))$ be the set of closed
subgroups of $\Isom(X)$. The Chabauty topology on $\mathcal{CL}$ is
generated by the basis of open subsets, for $C\in \mathcal{CL}$ a
closed subgroup, $K\subset \Isom(X)$ a compact subset and $U$ an open
neighborhood of the identity in $\Isom(X)$:
\[V_{K,U,C}:=\{D\in \mathcal{CL} \; | \; D\cap K \subset CU\textrm{ and }
  C\cap K \subset DU\}.\] Equipped with this topology, $\mathcal{CL}$ is
a compact space.

An important fact is that the subset of discrete subgroups of
$\Isom(X)$ is open in $\mathcal{CL}$ \cite[Prop 3.4]{BHK}. Following
the proof of this fact in \cite{BHK} actually yields a slightly
stronger statement, given below:
\begin{proposition}
  \label{prop:uniformly_discrete_open}
  Let $d$ be any metric inducing the compact-open topology on
  $\Isom(X)$. For any $R > 0$, the set of closed subgroups
  $G < \Isom(X)$ satisfying $\min_{g \in G \minus \{e\}} d(g, e) > R$
  is open in $\mathcal{CL}$.
\end{proposition}

We also have the following useful criterion for convergence in
$\mathcal{CL}$:
\begin{proposition}[{\cite[Proposition E.1.2]{BenedettiPetronio}}]
  A sequence $(C_n)_{n\in \N}$ of elements of $\mathcal{CL}$ converges
  to $C\in \mathcal{CL}$ if and only if both conditions below hold:
  \begin{enumerate}[label=(C\arabic*)]
    \item\label{item:Chabauty_C1} Any accumulation point of a sequence $(c_n)_{n\in \N}$, where each $c_n$ belongs to $C_n$, belongs to $C$.
    \item\label{item:Chabauty_C2} Each point of $C$ is the limit of a sequence $(c_n)_{n\in \N}$, where each $c_n$ belongs to $C_n$.
  \end{enumerate}
\end{proposition}

The notion of strong convergence combines algebraic convergence and
convergence in the Chabauty topology. We note that convergence in the
Chabauty topology is also classically called \emph{geometric
  convergence}.
\begin{definition}\label{def:strong_convergence}
  Let $(\rho_n)_{n\in\N}$ be a sequence in $\Hom(\G,\Isom(X))$.
  \begin{itemize}
  \item We say that the sequence $(\rho_n)$ converges \emph{strongly}
    to $\rho$ if it converges algebraically to $\rho$ and if the
    subgroups $\rho_n(\G)$ converge to $\rho(\G)$ in the Chabauty
    topology on the space of subgroups of $\Isom(X)$.
    
  \item If $\rho:\G \to \Isom(X)$ is geometrically finite, and
    $\mathcal H$ is the set of $\rho$-parabolic subgroups, we say the
    sequence $(\rho_n)$ converges \emph{relatively strongly} to $\rho$
    if it converges algebraically to $\rho$ and if, for any
    $H\in \mathcal H$, the sequence $\rho_n(H)$ converges to $\rho(H)$
    in the Chabauty topology.
  \end{itemize}
\end{definition}

\subsection{Equivalence between strong convergence and peripheral stability}

The proposition below links the notions of strong convergence for
geometrically finite representations and peripheral stability for EGF
representations:
\begin{proposition}
  \label{prop:strong_relative_convergence}
  Let $\rho\in \Hom(\G,\Isom(X))$ be a geometrically finite
  representation, and let $\mathcal H$ be the family of $\rho$-parabolic
  subgroups. Suppose that the sequence $(\rho_n)_{n\in \N}$ of
  representations in $\Hom(\G,\Isom(X))$ converges algebraically to
  $\rho$, and that the restriction of $\rho_n$ to each cusp group
  $H\in \mathcal H$ is faithful. Then the following are equivalent:
  \begin{enumerate}[label=(\arabic*)]
  \item\label{item:strong_convergence} The sequence $(\rho_n)$ converges strongly to $\rho$;
  \item\label{item:relative_strong_convergence} The sequence $(\rho_n)$ converges relatively strongly to $\rho$;
  \item\label{item:peripheral_stability} The family $\{\rho_n, n\in \N\}$ is a peripherally stable deformation of $\rho$.
  \end{enumerate}
\end{proposition}

The first implication \ref{item:strong_convergence} $\implies$
\ref{item:relative_strong_convergence} is not difficult and we prove
it below. The second implication
\ref{item:relative_strong_convergence} $\implies$
\ref{item:peripheral_stability} relies on Gromov hyperbolicity of the
space $X$. We will deal with it in the following \Cref{sec:CAT(-1)},
using the language of CAT($-1$) geometry. The last implication
\ref{item:peripheral_stability} $\implies$
\ref{item:strong_convergence} is more involved, since it relies on the
relative automaton discussed in \Cref{sec:automaton_exists}. We will
tackle it in \Cref{sec:peripheral_to_strong}.

\begin{proof}[Proof of \ref{item:strong_convergence}
  $\implies$ \ref{item:relative_strong_convergence}]
  We assume that $(\rho_n)$ converges strongly to $\rho$. Denote by
  $G_n$, resp. $G$, the images $\rho_n(\G)$, resp. $\rho(\G)$. Fix
  $H\in \mathcal H$, and let $H_n:=\rho_n(H)$ and
  $H_\infty:=\rho(H)$. Recall that $H$ is virtually nilpotent and
  infinite. Since $\rho_n$ is faithful on $H$, this implies that each
  $H_n$ preserves a subset of $\dX$ containing either one or two
  points.  Let $\{p_n^1, p_n^2\}$ be this subset (allowing for the
  possibility that $p_n^1 = p_n^2$).

  As $\rho_n$ converges algebraically to $\rho$, for each $h\in H$, we
  have that $\rho_n(h)\in H_n$ converges to $\rho(h)$ in
  $H_\infty$. So condition \ref{item:Chabauty_C2} for the Chabauty
  convergence of $H_n$ to $H_\infty$ is fulfilled. It also follows
  that both $p_n^1$ and $p_n^2$ converge to $p$. If
  $p_n^1 = p_n^2 = p_n$, this holds because $\rho_n(h)$ fixes $p_n$,
  and for any infinite-order $h \in H$ we know that $\rho(h)$ uniquely
  fixes $p$. Otherwise, $H$ is virtually infinite cyclic. So for every
  $h$ in a finite-index subgroup of $H$, we have
  $\rho_n(h)p_n^i = p_n^i$, and again when $h$ has infinite order the
  fixed points of $\rho_n(h)$ must converge to the unique fixed point
  of $\rho(h)$.

  We want to prove \ref{item:Chabauty_C1} for the convergence of $H_n$
  to $H_\infty$. So suppose that an element $g\in \Isom(X)$ is an
  accumulation point of some sequence $h_n$ in $H_n$. First, the
  Chabauty convergence of $G_n$ to $G$ ensures that $g\in
  G$. Moreover, as $h_n\cdot \{p_n^1, p_n^2\} = \{p_n^1, p_n^2\}$ for
  all $n$, we can pass to the limit to see that $g\cdot p=p$. But
  $H_\infty=\Stab_G(p)$, so that $g\in H_\infty$.
\end{proof}

\subsubsection{Relative strong convergence implies peripheral
  stability}\label{sec:CAT(-1)}

To prove the implication \ref{item:relative_strong_convergence}
$\implies$ \ref{item:peripheral_stability} in
\Cref{prop:strong_relative_convergence}, we need to use the fact that
the space $X$ is CAT($-1$), up to a rescaling of the metric. This
assumption is actually necessary, since the analogous implication does
\emph{not} hold if we instead only assume that $\rho$ is an EGF
representation into some higher-rank semisimple Lie group. Indeed,
\cite[Example 9.3]{Weisman-EGF} exhibits a continuous family
$\{\rho_t : 0 \le t \le \eps\}$ of representations of the free group
$\mathbf F_2 \simeq \Z * \Z$ into $\SL(4, \R)$ converging to an EGF
representation $\rho_0$, so that the family $\rho_t$ is relatively
strongly convergent but \emph{not} peripherally stable. In this
example, the limiting representation $\rho_0$ even has an injective
boundary extension, meaning it is a \emph{relative Anosov}
representation (another related generalization of geometrical
finiteness in higher rank).

In the cited example, the peripheral subgroups of $\mathbf F_2$ are
the conjugates of the cyclic free factors, and the restriction of
$\rho_t$ to each of these factors converges strongly to a unipotent
representation of $\Z$. The problem in the example occurs because the
limit set of $\rho_t(\Z)$ does not vary continuously at $t = 0$.

This problem no longer occurs in rank one, which follows from our next
lemma.  Recall that any virtually nilpotent subgroup $G$ of $\Isom(X)$
is either elliptic, parabolic or hyperbolic. For such a subgroup, we
denote by $\mathrm{Fix}(G)$ the set of fixed points in $\overline X$
of $G$ and by $\mathcal C(G)$ the convex hull in $\overline X$ of
$\mathrm{Fix}(G)\cup\Lambda_G$. 
\begin{lemma}\label{lem:CAT(-1).3}
  Let $G$ be a virtually nilpotent group and let $(\nu_n)$ be a
  sequence of representations $\nu_n : G \to \Isom(X)$, converging
  algebraically to a representation $\nu: G \to \Isom(X)$ whose image
  is a nontrivial parabolic subgroup.
  
  Then the convex hulls $\mathcal C(\nu_n(G))$ converge to
  $\Lambda_\nu$ in the Hausdorff topology on $\overline X$.
\end{lemma}

\begin{proof}
  First, since $\nu$ is nontrivial, $\nu_n$ is nontrivial for
  sufficiently large $n$ and has virtually nilpotent image. Thus
  $\nu_n(G)$ is a nontrivial elliptic, parabolic, or hyperbolic
  subgroup, implying that $\mathcal C(\nu_n(G))$ is nonempty.

  By assumption $\nu$ is parabolic, so $\Lambda_\nu$ is a singleton
  $\{p\}$, and there is an element $g \in G$ so that the unique fixed
  point of $\nu(g)$ in $\overline{X}$ is $p$. By algebraic convergence we
  know that $\nu_n(g) \to \nu(g)$, so any limit of fixed points of
  $\nu_n(g)$ in $\overline{X}$ is fixed by $\nu(g)$, hence equal to
  $p$. Since the fixed points of $\nu_n(g)$ contain the fixed points
  of $\nu_n(G)$, this proves that $\mathrm{Fix}(\nu_n(G)) \to \{p\}$
  and therefore $\mathcal{C}(\nu_n(G)) \to \{p\}$.
\end{proof}

Below we state the key lemma we need for our proof of the implication
at hand. This lemma says that isometries of $X$ whose fixed points all
lie close to $p$ and remain bounded far from $p$ belong to a compact
subset of $\Isom(X)$.

\begin{lemma}\label{lem:CAT(-1).2} Let $p\in \partial X$, and let $K$
  and $K'$ be two compact subsets of $\partial X$ not containing
  $p$. Then there exists an open neighborhood $U$ of $p$ in $\overline X$,
  disjoint from $K$ and $K'$, such that the set of isometries
  $g\in \Isom(X)$ satisfying $\mathrm{Fix}(g) \subset U$ and
  $g(K)\cap K'\neq \emptyset$ is relatively compact in $\Isom(X)$.
\end{lemma}

The idea behind this lemma is that isometries which fix points close
to $p$ should behave roughly like elements preserving a horosphere $S$
through $p$, and the stabilizer of $S$ in $G$ acts properly on
$S$.

Before proving the lemma, we set up some notation. Whenever
$\mathcal C$ is a convex subset of $\overline{X}$, and
$k \in \overline{X} \minus \mathcal C$, we let $[k; \mathcal{C}]$
denote the geodesic segment between $k$ and its projection
$\pi_{\mathcal{C}}(k)$ on $\mathcal C$.

Now fix a point $p \in \partial X$ as in the lemma, and let $S$ be a
horosphere centered at $p$. For any convex subset
$\mathcal C \subset \overline{X}$ and any
$k \in \overline{X} \minus \mathcal C$, let $s(k; \mathcal C)$ denote
the point on the intersection $[k; \mathcal C]\cap S$ closest to $k$
(assuming this intersection exists).

The first step in the proof of \Cref{lem:CAT(-1).2} is the following,
which says that these ``projections'' $s(k;\mathcal C)$ stay in a
compact subset of $X$ when $k$ stays far away from $p$ and
$\mathcal C$ is close to $p$.

\begin{lemma}\label{lem:CAT(-1).1}
  Let $K$ be a compact subset of $\partial X$ disjoint from $p$. Then
  there exists a neighborhood $U$ of $p$ in $\overline X$ such that:
  \begin{itemize}
  \item for any convex subset $\mathcal C$ of $U$ and any $k\in K$,
    the geodesic segment $[k;\mathcal C]$ intersects $S$ at least
    once, and
  \item the set
    $\{s(k;\mathcal C) \in X, k\in K, \mathcal C \subset U\}$ has
    compact closure in $X$.
  \end{itemize}
\end{lemma}
\begin{proof}
  For any $k\in K$, the geodesic $[k;\{p\}]$ intersects $S$
  transversely exactly once, at the point denoted by
  $s(k;\{p\})$. Moreover, the induced map
  $\partial X\setminus \{p\}\to S\setminus \{p\}$ is a homeomorphism,
  so the image of $K$ is compact in $S\setminus \{p\}\subset X$.

  Now, near any pair $(k;\mathcal C)$ such that the geodesic
  $[k;\mathcal C]$ intersects $S$ transversely, the assignment
  $(k;\mathcal C) \mapsto s(k;\mathcal C)$ is locally a well-defined
  continuous map, because the geodesic $[k;\mathcal C]$ varies
  continuously with $k$ and $\mathcal C$.  This proves the claim.
\end{proof}

\begin{proof}[Proof of \Cref{lem:CAT(-1).2}]
  Fix a neighborhood $U$ of $p$ in $\partial X$ that verifies the
  conclusion of the previous \Cref{lem:CAT(-1).1} for both $K$ and
  $K'$. Without loss of generality we may assume that $U$ is a convex
  subset of $\overline{X}$. Let $T\subset S\setminus \{p\}$ be a
  compact set containing all $s(k;\mathcal C)$ for $k\in K\cup K'$ and
  convex subsets $\mathcal C \subset U$, and let $D_T$ be the diameter
  of $T$.  Consider an isometry $g$ such that the convex hull
  $\mathcal C$ of its fixed points lies in $U$, and such that there
  exists $k\in K$ with $k'=g(k)\in K'$.  Then $g$ sends the geodesic
  $[k;\mathcal C]$ to $[k',\mathcal C]$. We want to prove that
  $g(s(k;\mathcal C))$ is uniformly close to $s(k';\mathcal C)$ on the
  geodesic $[k',\mathcal C]$.

  First, consider the case where $\mathcal C$ is a singleton in
  $U\subset \partial X$. Fix an origin $o\in X$. The stabilizer of $o$
  is compact and acts transitively on $\partial X$.  So, up to a
  uniformly compact conjugation, and if necessary a shrinking of $U$,
  we may assume without loss of generality that $\mathcal
  C=\{p\}$. Then $g$ stabilizes $S$, and we actually have
  $g(s(k;\mathcal C))=s(k';\mathcal C)$.

  The other possibility is that $\mathcal C$ is not a singleton, so
  that $\mathcal C \cap X$ is non-empty. In this case, we can define
  $t(k;\mathcal C)$ to be the distance from $s(k;\mathcal C)$ to the
  projection $\pi_{\mathcal{C}}(k)$ of $k$ on $\mathcal C$.  Up to
  shrinking $U$ another time to a smaller subset, we may assume that,
  for all $k\in K$, $\mathcal C$ included in $U$, we have
  $t(k;\mathcal C)> 2D_T$.

  For simplicity, for the rest of the proof, we let $s(k)$, $t(k)$ and
  $\pi(k)$ denote the points $s(k;\mathcal C)$, $t(k;\mathcal C)$ and
  $\pi_{\mathcal C}(k)$, for any $k\in K$. The points $\pi(k)$ and
  $\pi(k')$ are the projections on $\mathcal C$ of $s(k)$ and $s(k')$
  respectively, and the distance between those last two points is at
  most $D_T$.

  As the projection map to a convex set does not increase distances in
  a CAT$(-1)$-space, we obtain:
  \[ d(\pi(k),\pi(k'))\leq d(s(k),s(k'))\leq D_T.\] The distance
  $t(k')$ between $s(k')$ and its projection $\pi(k')$ then satisfies
  (by the triangle inequality):
  \begin{align*}
    |t(k')-t(k)|&\leq d(\pi(k'),\pi(k))+d(s(k),s(k'))\\
                & \leq 2D_T.
  \end{align*}
  So $g(s(k))$ and $s(k')$ are two points on the geodesic
  $[k',\mathcal C]$, whose distance from $\mathcal C$ is respectively
  $t(k)$ and $t(k')$. Thus their relative distance is less than
  $2D_T$.

  In either case above, we conclude that $g(s(k))$ lies at distance at
  most $2D_T$ from $s(k')\in T$. The $2D_T$-neighborhood $T'$ of $T$
  is a compact set in $X$. The isometry $g$ sends a point of $T$ to
  $T'$. As the action of $\Isom(X)$ on $X$ is proper, the set of
  isometries $h$ verifying $h(T)\cap T'\neq \emptyset$ is
  compact. This finishes the proof of the lemma.
\end{proof}

Equipped with these results, we can proceed with the proof that strong
convergence implies peripheral stability, under the assumptions of
\Cref{prop:strong_relative_convergence}.  Recall that we consider a
sequence $\rho_n$ of representations that converges algebraically to a
geometrically finite representation $\rho$. We assume moreover that
the representations $\rho_n$ are all faithful when restricted to the
parabolic peripheral subgroups.  We fix such a subgroup $H$ and, by
assumption \ref{item:relative_strong_convergence}, we have that
$H_n:=\rho_n(H)$ converges in the Chabauty topology to
$H_\infty:=\rho(H)$, which is a nontrivial parabolic subgroup in
$\Isom(X)$.

\begin{proof}[Proof of \ref{item:relative_strong_convergence} $\implies$
  \ref{item:peripheral_stability}]
  Suppose that the sequence $(\rho_n)$ is not peripherally stable and
  that the subgroup $H$ is a witness for this: there exists a
  neighborhood $U$ of $p$ in $\partial X$, a compact
  $K \subset \partial X$ disjoint from $U$ and a divergent sequence
  $\gamma_n\in H$ such that $\rho_n(\gamma_n)(K)$ has an accumulation
  point outside $U$.  So there is a compact $K'$ disjoint from $U$
  such that $\rho_n(g_n)(K)\cap K'\neq \emptyset$.

  Since $\rho(H)$ is a nontrivial parabolic subgroup and
  $\rho_n \to \rho$ algebraically, we may apply
  \Cref{lem:CAT(-1).3}. This means that $\Lambda_{H_n}$ converges to
  $\{p\}$ in the Hausdorff topology on $\partial X$. In particular,
  for sufficiently large $n$, $\Lambda_{H_n}$ is contained in $U$.

  Now we can apply \Cref{lem:CAT(-1).2}, which tells us that the
  elements $\rho_n(\gamma_n)$ remain in a compact subset of
  $\Isom(X)$. Up to extraction, we can assume they converge, and by
  the relative strong convergence assumption
  \ref{item:relative_strong_convergence}, the limit is $\rho(\gamma)$
  for some $\gamma\in H$. Thus $\rho_n(\gamma^{-1}\gamma_n)$ converges
  to the identity, and since $\rho(H)$ is discrete,
  \Cref{prop:uniformly_discrete_open} implies that $\rho_n(\gamma^{-1}\gamma_n) = e$, hence
  $\rho_n(\gamma) = \rho_n(\gamma_n)$, for sufficiently large
  $n$. Since $\rho_n$ is faithful when restricted to $H$, we have
  $\gamma_n = \gamma$, which contradicts the assumption that the
  sequence $\gamma_n$ is divergent. This proves that the family
  $(\rho_n)$ is peripherally stable around $\rho$.
\end{proof}

The last step of the proof of \Cref{prop:strong_relative_convergence}
is carried out in the following section.

\subsubsection{Peripheral stability implies strong
  convergence}\label{sec:peripheral_to_strong}

\begin{proof}[Proof of \ref{item:peripheral_stability} $\implies$
  \ref{item:strong_convergence}]

  The proof of this implication is similar to the proof of
  Proposition~4.6 in \cite{Weisman-EGF}. Assume that
  \ref{item:peripheral_stability} holds: the family $(\rho_n)$ is
  peripherally stable around $\rho$. Algebraic convergence
  $\rho_n \to \rho$ ensures that condition \ref{item:Chabauty_C2} for
  Chabauty convergence holds. To show that \ref{item:Chabauty_C1} also
  holds, it suffices to prove the following:
  \begin{claim}
    For any sequence $(\gamma_n)$ of pairwise distinct elements in
    $\G$, the sequence $(\rho_n(\gamma_n))$ leaves every compact
    subset of $\Isom(X)$.
  \end{claim}
  So, fix such a sequence $(\gamma_n)$ in $\G$. It will be enough to
  show that every subsequence of $(\gamma_n)$ has a further
  subsequence which leaves every compact subset of $\Isom(X)$, so we
  can extract subsequences throughout the rest of the argument.

  Fix a finite generating set $S$ for $\G$, and let $\Cay(\G, S, \Pi)$
  denote the \emph{relative Cayley graph} for $\G$, i.e. the Cayley
  graph of $\G$ with respect to the generating set
  \[
    S \cup \bigcup_{p \in \Pi} \Stab_\G(p).
  \]
  The path metric on this graph induces a metric on $\G$; note that
  this metric is \emph{not} quasi-isometric to the word metric induced
  by a finite generating set.

  We now consider two cases, depending on the behavior of our sequence
  $(\gamma_n)$ with respect to the metric coming from the relative
  Cayley graph:
  \begin{description}
  \item[Case 1: the sequence $(\gamma_n)$ is unbounded in
    $\Cay(\G, S, \Pi)$.]  For this case, we use the relative automaton
    $\mathcal G$ described in \Cref{sec:automaton_exists}. Using
    property \ref{item:automaton_quasidensity} of the automaton, for
    each $n$, write
    \[
      \gamma_n = \alpha_1^{(n)} \cdots \alpha_{m(n)}^{(n)}\beta^{(n)},
    \]
    where $\beta^{(n)} \in \G$ has length at most $R$ with respect to
    the path metric on $\Cay(\G, S)$, and each element
    $\alpha_i^{(n)}$ lies in a set $L(z_i)$ for a vertex path
    $z_1^{(n)} \to z_2^{(n)} \to \ldots \to z_{m(n) + 1}^{(n)}$ in the
    automaton $\mathcal{G}$. Since $\beta^{(n)}$ has uniformly bounded
    length, it suffices to prove that the sequence
    \[
      \rho_n(\gamma_n(\beta^{(n)})^{-1}) = \rho_n(\alpha_1^{(n)} \cdots
      \alpha_{m(n)}^{(n)})
    \]
    leaves every compact subset of $\Isom(X)$.

    Properties \ref{item:parabolic_vertex_coset} and
    \ref{item:conical_vertex_singleton} of the automaton tell us that
    each element $\alpha_i^{(n)}$ has uniformly bounded length with
    respect to the metric on $\Cay(\G, S, \Pi)$. So, our assumption
    for this case tells us that the length $m(n)$ of the vertex path
    is unbounded. After extracting a subsequence, we can assume that
    this length tends to infinity.

    Now, consider the sequence of subsets
    \begin{equation}
      \label{eq:nested_subset}
      \rho_n(\alpha_1^{(n)} \cdots \alpha_{m(n)}^{(n)})
      W(z^{(n)}_{m(n) + 1}).
    \end{equation}

    \Cref{lem:peripheral_edge_stability} implies that there is a
    uniform positive constant $\eps_0 > 0$ so that, for every
    sufficiently large $n$ and every every $1 \le i \le m(n)$, we have
    \begin{equation}
      \label{eq:uniform_strong_nesting}
      \rho_n(\alpha_i^{(n)})N_{\eps_0}(W(z^{(n)}_{i + 1})) \subset
      W(z_i^{(n)}).
    \end{equation}
    After extracting a subsequence, we can assume that the vertex
    $z_1^{(n)}$ is independent of $n$; we write this vertex as
    $z_1$. By applying \cite[Proposition 7.11]{Weisman-EGF}, we can
    see that the uniform strong nesting in
    \eqref{eq:uniform_strong_nesting} implies that, with respect to a
    fixed choice of metric on the open subset $W(z_1)$, the diameter
    of the set defined in \eqref{eq:nested_subset} tends to zero as
    $m(n)$ tends to infinity. (This step is where we apply property
    \ref{item:automaton_sets_proper} of the automaton, since otherwise
    the cited proposition in \cite{Weisman-EGF} does not apply.)

    We can extract a further subsequence so that the vertex
    $z^{(n)}_{m(n) + 1}$ does not depend on $n$, and write this vertex
    as $z'$. Then, as $W(z')$ is nonempty and open (so in particular
    has positive diameter), it follows that
    $\rho_n(\alpha_1^{(n)} \cdots \alpha_{m(n)}^{(n)})$ does not
    accumulate to any point in $\Homeo(\dX)$, and therefore leaves
    every compact subset of $\Isom(X)$ as desired.
  \item[Case 2: the sequence $(\gamma_n)$ is bounded in
    $\Cay(\G, S, \Pi)$.]  In this case, we choose a shortest
    representative for each $\gamma_n$ with respect to the relative
    generating set $S \cup \bigcup_{p \in \Pi}\Stab_\G(p)$. Such a
    word has the form of an alternating product
    \[
      \gamma_n = g_0^{(n)}h_1^{(n)}g_1^{(n)} \cdots
      h_{k(n)}^{(n)}g_{k(n)}^{(n)},
    \]
    where, for each fixed $0 \le i \le k$, $(g_i^{(n)})$ is a bounded
    sequence in $\G$ (with respect to the word metric induced by $S$),
    and each $h_i^{(n)}$ lies in $\Stab_\G(p_i^{(n)})$ for some
    $p_i^{(n)}$ in $\Pi$. Since the length of this word is uniformly
    bounded, we can extract a subsequence and assume that $k(n)$ is a
    fixed number $k$, independent of $n$. By repeatedly combining
    adjacent terms in this word and extracting further subsequences,
    we can also assume that for each fixed $1 \le i \le k$, the
    sequence $(h_i^{(n)})$ is unbounded (with respect to the word
    metric coming from $S$), the parabolic point $p_i^{(n)}$ is a
    point $p_i$ independent of $n$, and
    $\rho(g_{i})p_{i} \ne p_{i-1}$.

    Let $K$ be a compact subset of
    $\dX \minus \{\rho(g_k^{-1})p_k\}$ with nonempty
    interior. We claim that the sequence $\rho_n(\gamma_n)K$ converges
    to a singleton; since $K$ has nonempty interior this will ensure
    that $\rho_n(\gamma_n)$ leaves every compact subset of
    $\Homeo(\dX)$, hence of $\Isom(X)$.

    To prove the claim, we induct on $k$. In the case $k = 1$, since
    $K$ is a compact subset of
    $\dX \minus \{\rho(g_1^{-1})p_1\}$, and $g_1$ is fixed, we
    know that for all sufficiently large $n$, the set $\rho_n(g_1)K$
    is a compact subset of $\dX \minus \{p_1\}$. Then, since
    $h_1^{(n)}$ is unbounded in $\Stab_\G(p_1)$, the
    peripheral stability assumption implies that for any neighborhood
    $U$ of $p_1$ in $\dX$, for all sufficiently large $n$ we have
    $\rho_n(h_1^{(n)}g_1)K \subset U$. So, $\rho_n(h_1^{(n)}g_1)K$
    converges to the singleton $\{p_1\}$, and so the sequence
    $\rho_n(g_0h_1^{(n)}g_1)K$ converges to the singleton
    $\{g_0p_1\}$.

    When $k > 1$, the exact same reasoning implies that the sequence
    of sets $\rho_n(h_k^{(n)}g_k)K$ converges to the singleton
    $\{p_k\}$. In particular, since we know that
    $\rho(g_{k-1})p_k \ne p_{k-1}$, we know that for
    sufficiently large $n$ we also have
    $\rho_n(g_{k-1})p_k \ne p_{k-1}$, and thus (also for large $n$)
    the set $\rho_n(h_k^{(n)}g_k)K$ lies in a fixed compact subset of
    $\dX \minus \{\rho(g_{k-1}^{-1})p_{k-1}\}$. So applying
    induction we see that the sequence of sets
    \[
      \rho_n(\gamma_n)K = \rho_n(g_0h_1^{(n)} \cdots
      h_{k-1}^{(n)}g_{k-1})\rho_n(h_k^{(n)}g_k)K
    \]
    again converges to a singleton, and we are done.
  \end{description}
\end{proof}

\section{Convergence of limit sets and Cannon-Thurston
  maps}\label{ssec:convergencetheorems}

We are now able to prove \Cref{thm:theorem_1} and \Cref{thm:theorem_2}
from the introduction. We have already done the difficult part, which
was to connect the framework of EGF representations and peripheral
stability to the notions of geometrical finiteness and strong
convergence; with this relationship established, both of our main
theorems are straightforward corollaries of the relative stability
theorem for EGF representations
(\Cref{thm:relative_stability_theorem}).

\subsection{Limit sets}

First we will prove \Cref{thm:theorem_1}, whose statement is subsumed
by the following:
\begin{theorem}[Convergence of limit sets]
  \label{thm:limit_limit_sets}
  Let $X$ be a noncompact rank-one symmetric space, let $\G$ be a
  finitely generated group, and let $(\rho_n)_{n \in \N}$ be a
  sequence of faithful representations converging algebraically to a
  geometrically finite representation $\rho$.

  Then, $\rho_n$ converges to $\rho$ strongly if and only if
  $\rho_n \to \rho$ relatively strongly. In this case:
  \begin{enumerate}[label=(\arabic*)]
  \item\label{item:geom_finite_rel_open} For all sufficiently large
    $n$, the representation $\rho_n$ is geometrically finite. Further,
    for any subgroup $H \subset \Gamma$, $\rho_n(H)$ is parabolic only
    if $\rho(H)$ is parabolic.
  \item\label{item:limsets_converge} The limit sets $\Lambda_n$ of
    $\rho_n(\Gamma)$ converge in the Hausdorff topology to the limit
    set $\Lambda$ of $\rho(\Gamma)$.
  \end{enumerate}
\end{theorem}
\begin{proof}
  The equivalence of strong convergence and strong relative
  convergence for $\rho_n \to \rho$ is given by
  \Cref{prop:strong_relative_convergence}. So, now suppose that the
  convergence $\rho_n \to \rho$ is relatively strong.

  We first show that \ref{item:geom_finite_rel_open} holds. As we have
  already observed several times, the limiting representation $\rho$
  is EGF with respect to the collection $\mathcal H$ of its $\rho$-parabolic
  subgroups, which are virtually nilpotent. Since $(\rho_n)$ converges
  relatively strongly and $\rho_n$ is faithful, the family $(\rho_n)$
  is peripherally stable, with respect to $\mathcal H$ around $\rho$,
  by \Cref{prop:strong_relative_convergence}. Then by
  \Cref{thm:relative_stability_theorem}, $\rho_n$ is EGF for
  sufficiently large $n$, again with respect to the collection
  $\mathcal H$. Thus, by \Cref{prop:egf_nilpotent_gf}, the
  representations $\rho_n$ are geometrically finite for large $n$, and
  the collection $\mathcal H_n$ of $\rho_n$-parabolic subgroups is contained in
  $\mathcal{H}$. This proves \ref{item:geom_finite_rel_open}.

  To prove \ref{item:limsets_converge}, let $\epsilon > 0$ be
  fixed. One may write the limit set $\Lambda$ of $\rho(\Gamma)$ as a
  union of finitely many compact subsets with diameter at most
  $\epsilon$, with respect to a chosen visual metric on $\partial
  X$. Applying the boundary extension $\phi:\Lambda \to \bgamh$ to
  each of these sets, we obtain a finite collection $\mathcal{Z}$ of
  compact subsets of $\bgamh$, so that
  $\bigcup_{Z \in \mathcal{Z}} Z = \bgamh$, and so that each
  $\phi^{-1}(Z)$ has diameter at most $\epsilon$.

  Now let $\phi_n:\Lambda_n \to \bgamh$ be the boundary extension for
  the EGF representation $\rho_n$, with respect to the peripheral
  structure $\mathcal{H}$; recall from
  \Cref{prop:good_boundary_extension} that the domain of this boundary
  extension is the limit set $\Lambda_n$ of $\rho_n(\Gamma)$. Applying
  the inclusion in \Cref{thm:relative_stability_theorem}, we see that
  for all sufficiently large $n$ and all $Z \in \mathcal{Z}$, the set
  $\phi_n^{-1}(Z)$ is contained in an $\epsilon$-neighborhood of
  $\phi^{-1}(Z)$, and thus
  $\Lambda_n = \bigcup_{Z \in \mathcal{Z}}\phi_n^{-1}(Z)$ is within
  Hausdorff distance $2\epsilon$ of
  $\Lambda = \bigcup_{Z \in \mathcal{Z}}\phi^{-1}(Z)$. Since
  $\epsilon > 0$ was arbitrary this completes the proof.
\end{proof}

\subsection{Cannon-Thurston maps}
\label{sec:ct_maps}

We now turn to \Cref{thm:theorem_2}, which concerns the convergence of
Cannon-Thurston maps for a strongly convergent sequence $(\rho_n)$ of
geometrically finite representations. In order to define
Cannon-Thurston maps, we need the notion of a weakly type-preserving
sequence of representations:
\begin{definition}
  \label{defn:type_preserving}
  A sequence of representations $(\rho_n)_{n\in \N}$ from $\G$ to
  $\Isom(X)$ is \emph{weakly type-preserving} if the collection
  $\mathcal H_n$ of $\rho_n$-parabolic subgroups of $\G$ always
  contains the collection $\mathcal H_1$ of $\rho_1$-parabolic
  subgroups of $\G$.
\end{definition}

Suppose that a sequence of representations $(\rho_n)$ converges
relatively strongly to a geometrically finite representation $\rho$,
whose collection of $\rho$-parabolic subgroups is $\mathcal H$. By
\Cref{thm:limit_limit_sets}, after forgetting a finite number of
indices, we may assume that each $\rho_n$ is geometrically finite, and
that the collection of $\rho_n$-parabolic subgroups $\mathcal{H}_n$ is a 
subset of $\mathcal{H}$. Since each $\mathcal{H}_n$ and
$\mathcal{H}$ consists of a finite number of conjugacy classes, after
extracting a further subsequence we can assume that
$\mathcal{H}_1 \subseteq \mathcal{H}_n$ for all $n$. Thus we have
shown:
\begin{proposition}
  \label{prop:subseq_type_preserve}
  Suppose that $\rho$ is a faithful geometrically finite
  representation and $(\rho_n)$ is a sequence of faithful
  representations converging strongly to $\rho$. Then a subsequence of
  $(\rho_n)$ is weakly type-preserving, and consists of geometrically
  finite representations.
\end{proposition}

Now let $(\rho_n)$ be any weakly type-preserving sequence of
faithful geometrically finite representations, and let $\mathcal{H}_n$
be the collection of $\rho_n$-parabolic subgroups. The inclusion
$\mathcal{H}_1 \subseteq \mathcal{H}_n$ implies (via
\Cref{prop:egf_nilpotent_gf}) that the representation $\rho_1$ is
actually EGF with respect to $\mathcal{H}_n$ for every $n$. Thus we
have a $\G$-equivariant boundary extension
$\phi_{1,n}:\Lambda_1 \to \partial(\G, \mathcal{H}_n)$. We also know
from \Cref{thm:injective_bdry_extension} that the boundary extension
$\phi_{n,n}:\Lambda_n \to \partial(\G, \mathcal{H}_n)$ for $\rho_n$ is
a $\G$-equivariant homeomorphism. So we may make the following
definition:
\begin{definition}
  For a weakly type-preserving sequence $(\rho_n)$ of faithful
  geometrically finite representations, we define the Cannon-Thurston
  maps $\CT_{1,n}:\Lambda_1 \to \Lambda_n$ by the composition
  $\phi_{n,n}^{-1} \circ \phi_{1,n}$.
\end{definition}

If the sequence $(\rho_n)$ converges strongly to a geometrically
finite representation $\rho$, then by \Cref{thm:limit_limit_sets},
$\mathcal{H}_n$ is eventually a subset of the collection $\mathcal{H}$
of parabolic subgroups for $\rho$. In particular, as
$\mathcal{H}_1 \subseteq \mathcal{H}_n$ for all $n$, we have
$\mathcal{H}_1 \subseteq \mathcal{H}$. Thus (again by
\Cref{prop:egf_nilpotent_gf}) there is an EGF boundary extension
$\phi_{1, \infty}:\Lambda_1 \to \bgamh$, and we may also define the
Cannon-Thurston map $\CT_{1,\infty}:\Lambda_1 \to \Lambda$ via the
composition $\phi^{-1} \circ \phi_{1,\infty}$, where
$\phi:\Lambda \to \bgamh$ is the EGF boundary extension for $\rho$.

With this notation established, we can now prove \Cref{thm:theorem_2}
from the introduction; we give a restatement of this result
below. This result can be thought of as a more precise version of the
convergence of limit sets expressed in \Cref{thm:limit_limit_sets}.
\begin{theorem}[Convergence of Cannon-Thurston maps]
  Let $X$ be a noncompact rank-one symmetric space, let $\G$ be a
  finitely generated group, and let $(\rho_n)_{n \in \N}$ be a weakly
  type-preserving sequence of faithful geometrically finite
  representations, converging relatively strongly to a geometrically
  finite representation $\rho$.

  Then the sequence of Cannon-Thurston maps
  $\CT_{1,n}:\Lambda_1 \to \Lambda_n$ converges uniformly to the
  Cannon-Thurston map $\CT_{1,\infty}:\Lambda_1 \to \Lambda_n$.
\end{theorem}
\begin{proof}
  First, if $\rho_1$ is elementary, meaning $\Lambda_1$ contains one or two points, then the statement is easy. We assume from now on that $\rho_1$ is non-elementary.

  Using the notation established above, we first observe that, since
  $\mathcal{H}_n \subseteq \mathcal{H}$ for all sufficiently large
  $n$, we can once again apply \Cref{prop:egf_nilpotent_gf} to see
  that that $\rho_n$ is EGF with respect to $\mathcal{H}$. Thus there
  is an EGF boundary extension $\phi_{n,\infty}:\Lambda_n \to \bgamh$
  and a Cannon-Thurston map
  $\CT_{n,\infty}:\Lambda_n \to \Lambda_\infty$ given by
  $\phi^{-1} \circ \phi_{n,\infty}$.

  We claim that the composition $\CT_{n,\infty} \circ \CT_{1,n}$ is
  precisely the Cannon-Thurston map $\CT_{1, \infty}$. As $\rho_1(\G)$
  is non-elementary, $\Lambda_1$ contains at least three points. Then
  there is some $\gamma \in \Gamma$ which has an attracting fixed
  point for the $\rho_1$-action of $\G$ on $\Lambda_1$. Any
  $(\rho_1, \rho)$-equivariant continuous map $\Lambda_1 \to \Lambda$
  must take this attracting fixed point to the attracting fixed point
  of $\rho(\gamma)$ (if $\rho(\gamma)$ is loxodromic) or to the unique
  fixed point of $\rho(\gamma)$ (if $\rho(\gamma)$ is
  parabolic). Moreover, since $\rho_1(\G)$ is non-elementary, every
  $\rho_1(\G)$-orbit in $\Lambda_1$ is dense. So any equivariant
  continuous map $\Lambda_1 \to \Lambda$ is uniquely determined on a
  dense set and therefore must agree with the Cannon-Thurston
  map. Since the composition $\CT_{n, \infty} \circ \CT_{1,n}$ is such
  an equivariant continuous map, we obtain
  \[
    \CT_{n, \infty} \circ \CT_{1,n} = \CT_{1, \infty},
  \]
  as claimed. We sum up the situation in the commutative diagram of \Cref{fig:diagram}.
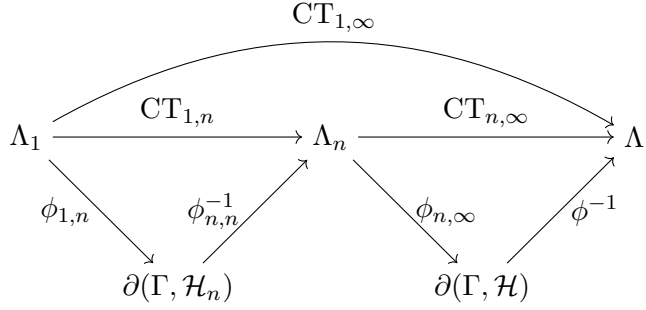
\begin{figure}[ht]
  \begin{center}
  \begin{tikzpicture}[
    ]
   %Nodes
    \node   (L1)  at   (0,0)    {$\Lambda_1$};
    \node   (Ln)  at   (4,0)    {$\Lambda_n$};
    \node   (Bn)  at   (2,-2)   {$\partial (\G,\mathcal H_n)$};
    \node   (B)  at   (6,-2)   {$\partial (\G,\mathcal H)$};
    \node   (L)   at   (8,0)    {$\Lambda$};
    
   %Lines
    \draw[->] (L1) -- (Bn) node [midway,left] {$\phi_{1,n}$};
    \draw[->] (Bn) -- (Ln) node [midway,left] {$\phi^{-1}_{n,n}\,$};
    \draw[->] (L1) -- (Ln) node [midway,above] {$\CT_{1,n}$};
    \draw[->] (Ln) -- (L) node [midway,above] {$\CT_{n,\infty}$};

    \draw[->] (Ln) -- (B) node [midway,right] {$\phi_{n,\infty}$};
    \draw[->] (B) -- (L) node [midway,right] {$\phi^{-1}$};
    \draw[->] (L1) to[bend left=30]node [midway,above] {$\CT_{1,\infty}$} (L) ;
  \end{tikzpicture}
  \end{center}
  \caption{Synthetic view of boundary extensions and Cannon-Thurston
    maps.}\label{fig:diagram}
 \end{figure}

  Now, let $\epsilon > 0$ be fixed. As in the proof of
  \Cref{thm:theorem_1}, we may cover $\bgamh$ with a finite collection
  $\mathcal{Z}$ of compact sets so that, for each $Z \in \mathcal{Z}$,
  the preimage $\phi^{-1}(Z)$ has diameter at most $\epsilon$. By
  \Cref{thm:relative_stability_theorem}, for all sufficiently large
  $n$ (depending only on $\epsilon$ and $\mathcal{Z}$), for every
  $Z \in \mathcal{Z}$ the preimage $\phi_{n, \infty}^{-1}(Z)$ lies in
  an $\epsilon$-neighborhood of $\phi^{-1}(Z)$. Then, for any
  $x \in \Lambda_1$, we may choose $Z \in \mathcal{Z}$ so that
  $\phi_{1,\infty}(x) \in Z$, and therefore
  $\CT_{1, \infty}(x) = \phi^{-1} \circ \phi_{1,\infty}(x) \in
  \phi^{-1}(Z)$.

  On the other hand, since
  $\phi_{1, \infty} = \phi \circ \CT_{1, \infty}(x) \in Z$, we have
  \[
    \phi \circ \CT_{n, \infty} \circ \CT_{1, n}(x) \in Z,
  \]
  or equivalently $\phi_{n, \infty} \circ \CT_{1,n}(x) \in
  Z$. Thus $\CT_{1,n}(x) \in \phi_{n, \infty}^{-1}(Z)$ and the
  distance between $\CT_{1,n}(x)$ and $\CT_{1, \infty}(x)$ is at most
  $2\epsilon$, independent of $x$.

\end{proof}

\bibliographystyle{alpha}
\bibliography{biblio}

\end{document}